\documentclass{amsart}
\usepackage{amssymb, amsmath, mathrsfs, verbatim, url, bbm}

\theoremstyle{plain}
\newtheorem{theorem}{Theorem}[section]
\newtheorem{lemma}[theorem]{Lemma}
\newtheorem{proposition}[theorem]{Proposition}
\newtheorem{corollary}[theorem]{Corollary}
\newtheorem*{claim}{Claim}
\newcommand{\MainTheoremName}{Main Theorem}

\theoremstyle{definition}
\newtheorem{definition}[theorem]{Definition}
\newtheorem{example}[theorem]{Example}

\theoremstyle{remark}

\newtheorem{question}[theorem]{Question}

\numberwithin{equation}{section}

\DeclareMathOperator{\cf}{cf}
\DeclareMathOperator{\ci}{ci}

\newcommand{\nbd}{\nobreakdash}

\newcommand{\card}[1]{\lvert #1\rvert}

\newcommand{\la}{\langle}
\newcommand{\ra}{\rangle}

\newcommand{\inv}[1]{#1^{-1}}
\newcommand{\closure}[1]{\overline{#1}}
\newcommand{\restrict}{\upharpoonright}
\newcommand{\density}[1]{d(#1)}
\newcommand{\weight}[1]{w(#1)}
\newcommand{\ow}[1]{Nt(#1)}

\newcommand{\piweight}[1]{\pi(#1)}

\newcommand{\naturals}{\mathbb{N}}
\newcommand{\rationals}{\mathbb{Q}}

\newcommand{\integers}{\mathbb{Z}}

\newcommand{\wma}{we may assume}

\newcommand{\op}{\mathrm{op}}

\newcommand{\omegaop}{$\omega^\op$\nbd-like}
\newcommand{\kappaop}{$\kappa^\op$\nbd-like}
\newcommand{\cardop}[1]{$#1^\op$\nbd-like}
\newcommand{\elemsub}{\prec}

\newcommand{\lots}{linearly ordered topological space}

\newcommand{\betann}{\beta\naturals\setminus\naturals}

\newcommand{\tlex}{\mathrm{lex}}
\newcommand{\locpct}{linearly ordered compactum}
\newcommand{\locpcta}{linearly ordered compacta}

\newcommand{\Locpcta}{Linearly ordered compacta}

\newcommand{\mcA}{\mathcal{ A}}
\newcommand{\mcB}{\mathcal{ B}}
\newcommand{\mcC}{\mathcal{ C}}
\newcommand{\mcD}{\mathcal{ D}}
\newcommand{\mcE}{\mathcal{ E}}
\newcommand{\mcF}{\mathcal{ F}}

\newcommand{\mcU}{\mathcal{ U}}
\newcommand{\mcV}{\mathcal{ V}}
\newcommand{\mcW}{\mathcal{ W}}

\newcommand{\om}{\omega}
\newcommand{\oml}{\om_1}
\newcommand{\ie}{{\it i.e.},}
\newcommand{\eg}{{\it e.g.},}
\newcommand{\many}{\nbd-many}
\newcommand{\bow}[1]{b\ow{X}}

\begin{document}

\title{GO-spaces and Noetherian spectra}

\author{David Milovich}


\address{Texas A\&M International University\\ Dept. of Engineering, Mathematics, and Physics\\ 5201 University Blvd.\\ Laredo, TX 78041}
\email{david.milovich@tamiu.edu}
\thanks{Some of these results are from the author's thesis, which was partially supported by an NSF graduate fellowship.}

\begin{abstract}
The Noetherian type of a space is the least $\kappa$ for which the space has a \kappaop\ base, \ie\ a base in which no element has $\kappa$\many\ supersets.  We prove some results about Noetherian types of (generalized) ordered spaces and products thereof.  For example: the density of a product of not-too-many compact linear orders never exceeds its Noetherian type, with equality possible only for singular Noetherian types; we prove a similar result for products of Lindel\"of GO-spaces.  A countable product of compact linear orders has an \cardop{\omega_1}\ base if and only if it is metrizable, and every metrizable space has an \omegaop\ base.  An infinite cardinal $\kappa$ is the Noetherian type of a compact LOTS if and only if $\kappa\not=\omega_1$ and $\kappa$ is not weakly inaccessible.  There is a Lindel\"of LOTS with Noetherian type $\oml$ and there consistently is a Lindel\"of LOTS with weakly inaccessible Noetherian type.
\end{abstract}

\maketitle

\section{Introduction}

The Noetherian type of a topological space is an order\nbd-theoretic analog of its weight.  

\begin{definition}
Given a cardinal $\kappa$, define a poset to be $\kappa^{\op}$\nbd-\emph{like} if no element is below $\kappa$\nbd-many elements.
\end{definition}

In the context of families of subsets of a topological space, we will always implicitly order by inclusion.  For example, a descending chain of open sets of type $\om$ is \omegaop; an ascending chain of open sets of type $\om$ is \cardop{\oml}\ but not \omegaop.

\begin{definition} Given a space $X$, 
\begin{itemize}
\item the {\it weight} of $X$, or $\weight{X}$, is the least $\kappa\geq\omega$ such that $X$ has a base of size at most $\kappa$;
\item the \emph{Noetherian type} of $X$, or $\ow{X}$, is the least $\kappa\geq\omega$ such that $X$ has a base that is \kappaop.
\end{itemize}
\end{definition}
Equivalently, $\ow{X}$ is the least $\kappa\geq\om$ such that $X$ has a base $\mcB$ such that $\bigcap\mcA$ has empty interior for all $\mcA\in[\mcB]^\kappa$.

Noetherian type was introduced by Peregudov~\cite{peregudov97}.  
Preceding this introduction are several papers by Peregudov, \v Sapirovski\u\i\, and Malykhin~\cite{malykhin,peregudov76c,peregudov76p,peregudovS76} about the topological properties $\ow{\cdot}=\om$ and $\ow{\cdot}\leq\oml$.  
More recently, the author has extensively investigated the Noetherian type of $\betann$~\cite{milovichnts} and the Noetherian types of homogeneous compacta and dyadic compacta~\cite{milovichnth}.  
(See Engelking~\cite{engelking}, Juh\'asz~\cite{juhasz}, and Kunen~\cite{kunenst} for all undefined terms.)

A surprising result from~\cite{milovichnth} is that no dyadic compactum has Noetherian type $\oml$.  
In other words, given an \cardop{\oml}\ base of a dyadic compactum $X$, one can construct an \omegaop\ base of $X$.  
This result does not generalize to all compacta.  
In that same paper, it was shown how to construct a compactum with Noetherian type $\kappa$, for any infinite cardinal $\kappa$.  
It is still an open problem whether any infinite cardinals other than $\oml$ are excluded from the spectrum of Noetherian types of dyadic compacta, although it was shown that there are dyadic compacta with Noetherian $\om$ and dyadic compacta with Noetherian type $\kappa^+$, for every infinite cardinal $\kappa$ with uncountable cofinality.

\begin{question}
If $\kappa$ is a singular cardinal with cofinality $\omega$, then is there a dyadic compactum with Noetherian type $\kappa^+$?  
Is there a dyadic compactum with weakly inaccessible Noetherian type?
\end{question}

The above two questions are typical of the ``sup=max'' problems of set\nbd-theoretic topology.  See Juh\'asz~\cite{juhasz} for a systematic study of these problems.

Though the above two questions remain open problems, we can now answer the corresponding questions for compact linear orders.  The spectrum of Noetherian types of linearly ordered compacta includes $\om$, excludes $\oml$, includes all singular cardinals, includes $\kappa^+$ for all uncountable cardinals $\kappa$, and excludes all weak inacessibles.  In the process of proving this claim, we will prove a general technical lemma which says roughly that if $X$ is a product of not-too-many $\mu$\nbd-compact GO-spaces for some fixed cardinal $\mu$, then $\density{X}\leq\ow{X}$ and in most cases $\density{X}<\ow{X}$.  

\begin{definition}\
\begin{itemize}
\item A space $X$ is $\kappa$\nbd-compact if $\kappa$ is a cardinal and every open cover of $X$ has a subcover of size less than $\kappa$.
\item A {\it GO-space}, or generalized ordered space, is a subspace of a \lots.  Equivalently, a GO-space is a linear order with a topology that has a base consisting only of convex sets.
\item The {\it density} $\density{X}$ of a space $X$ is the least infinite cardinal $\kappa$ such that $X$ has a dense subset of size at most $\kappa$.
\end{itemize}
\end{definition}

It is natural to ask what happens to the spectrum of Noetherian types of compact linear orders if we gently relax the assumption of compactness.  
It turns out that there are Lindel\"of linear orders with Noetherian type $\oml$, and, less expectedly, that it is consistent (relative to existence of an inaccessible cardinal) that there is a Lindel\"of linear order with weakly inaccessible Noetherian type.  
However, it is not consistent for a Lindel\"of GO-space to have strongly inacessible Noetherian type.

We also consider the relationship between metrizability and Noetherian type, focusing on GO-spaces.  
Every metric space has Noetherian type $\om$.  
For a Lindel\"of GO-space $X$, $X$ is metrizable if and only if $\ow{X}=\om$ if and only if $\ow{X}=\oml$ and $X$ is separable.  
For a countable product $X$ of compact linear orders, $X$ is metrizable if and only if $\ow{X}=\om$ if and only if $\ow{X}=\oml$.  (Note that every Lindel\"of metric space is separable, and every compact GO-space is a compact linear order.)

\section{Small densities and large Noetherian types}

\begin{definition} The {\it $\pi$\nbd-weight} $\piweight{X}$ of a space $X$ is the least infinite cardinal $\kappa$ such that a space has $\pi$\nbd-base of size at most $\kappa$;
\end{definition}

\begin{proposition}\label{PROowpiw}\cite{peregudov97}
If $X$ is a space and $\piweight{X}<\cf\kappa\leq\kappa\leq\weight{X}$, then $\ow{X}>\kappa$.
\end{proposition}
\begin{proof}
Suppose $\mcA$ is a base of $X$ and $\mcB$ is $\pi$\nbd-base of $X$ of size at most $\piweight{X}$.  
We then have $\card{\mcA}\geq\kappa$; hence, there exist $\mcU\in[\mcA]^\kappa$ and $V\in\mcB$ such that $V\subseteq\bigcap\mcU$.  
Hence, there exists $W\in\mcA$ such that $W\subseteq V\subseteq\bigcap\mcU$; hence, $\mcA$ is not \kappaop.
\end{proof}

Note that if $X$ is a product of at most $\density{X}$\many\ GO-spaces, then $\piweight{X}=\density{X}$ is witnessed by the following construction.  
For any $D$ (topologically) dense in $X$ and of minimal size, collect all the finitely supported products of topologically open intervals with endpoints from the union of $\{\pm\infty\}$ and the set of all coordinates of points from $D$.

Trivially, $\ow{X}\leq\weight{X}^+$ for all spaces $X$.  The next example shows that this upper bound is attained.

\begin{example}\cite{milovichnth}
The double\nbd-arrow space, defined as $((0,1]\times\{0\})\cup([0,1)\times\{1\})$ ordered lexicographically, has $\pi$\nbd-weight $\omega$ and weight $2^{\aleph_0}$.  By Proposition~\ref{PROowpiw}, it has Noetherian type $\left(2^{\aleph_0}\right)^+$.
\end{example}

\section{Lindel\"{o}f GO-spaces}\label{SEClots}


\begin{theorem}\label{THMallmetric}
Every metric space has an \omegaop\ base.
\end{theorem}
\begin{proof}
Let $X$ be a metric space.  For each $n<\omega$, let $\mcA_n$ be a locally finite open refinement of the (open) balls of radius $2^{-n}$ in $X$.  Set $\mcA=\bigcup_{n<\omega}\mcA_n\setminus\{\varnothing\}$.  The set $\mcA$ is a base of $X$ because if $p\in X$ and $n<\omega$, then there exists $U\in\mcA_{n+1}$ such that $p\in U$ and $U$ is contained in the ball of radius $2^{-n}$ with center $p$.  Let us show that $\mcA$ is \omegaop.  Suppose that $m<\omega$, $U\in\mcA$, $V\in\mcA_m$, and $U\subsetneq V$.  There then exist $p\in U$ and $\epsilon_0>\epsilon_1>0$ such that the $\epsilon_0$\nbd-ball with center $p$ is contained in $U$ and the $\epsilon_1$\nbd-ball with center $p$ intersects only finitely many elements of $\mcA_n$ for all $n<\omega$ satisfying $2^{-n}>\epsilon_0/2$.  If $2^{-m}\leq\epsilon_0/2$, then $V$ is contained in the $\epsilon_0$\nbd-ball with center $p$, in contradiction with $U\subsetneq V$.  Hence, $2^{-m}>\epsilon_0/2$; hence, there are only finitely many possibilities for $m$ and $V$ given $U$, for $V$ intersects the $\epsilon_1$\nbd-ball with center $p$.
\end{proof}

\begin{lemma}\label{LEMlocfinlindeloflots}
Let $X$ be a Lindel\"{o}f GO-space with open cover $\mcA$.  The cover $\mcA$ has a countable, locally finite refinement consisting only of countable unions of open convex sets.
\end{lemma}
\begin{proof}
Let $\{A_n:n<\omega\}$ be a countable refinement of $\mcA$ consisting only of open convex sets.  For each $n<\omega$, set $B_n=A_n\setminus\bigcup_{m<n}A_m$; set $\mcB=\{B_n:n<\omega\}$.  The set $\mcB$ is a locally finite refinement of $\mcA$.  Let $\mcC$ be the set of open convex subsets of $X$ which intersect only finitely many elements of $\mcB$.  Let $\mcD$ be the set of $U\in\mcC$ satisfying $\closure{U}\subseteq V$ for some $V\in\mcC$.  Let $\{D_n:n<\omega\}$ be a countable subcover of $\mcD$.  For each $n<\omega$, set $E_n=\closure{D_n\setminus\bigcup_{m<n}D_m}$; set $\mcE=\{E_n:n<\omega\}$.  The set $\mcE$ is a locally finite refinement of $\mcC$.  For each $n<\omega$, set $F_n=A_n\setminus\bigcup\{E\in\mcE:B_n\cap E=\varnothing\}$, which is a countable union of convex sets; set $\mcF=\{F_n:n<\omega\}$.  Since $\mcE$ is locally finite, each $F_n$ is open.  Hence, each $F_n$ is a countable union of open convex sets.  Moreover, $B_n\subseteq F_n\subseteq A_n$ for all $n<\omega$; hence, $\mcF$ is a refinement of $\mcA$.  

Thus, it suffices to show that $\mcF$ is locally finite.  Since $\mcE$ is a locally finite cover of $X$, it suffices to show that each element of $\mcE$ only intersects finitely many elements of $\mcF$.  Let $i<\omega$ and choose $V\in\mcC$ such that $E_i\subseteq V$.  Suppose $j<\omega$ and $E_i\cap F_j\not=\varnothing$.  We then have $E_i\cap B_j\not=\varnothing$ by definition of $F_j$.  Hence, $V\cap B_j\not=\varnothing$; hence, there are only finitely many possibilities for $B_j$; hence, there are only finitely many possibilities for $F_j$.
\end{proof}

\begin{definition}
Let $H_\theta$ denote the set of all sets that are hereditarily of size less $\theta$, where $\theta$ is a regular cardinal sufficiently large for the argument at hand.  The relation $M\elemsub H_\theta$ means that $\la M,\in\ra$ is an elementary substructure of $\la H_\theta,\in\ra$.
\end{definition}

\begin{lemma}\label{LEMnonseplots}
Let $X$ be a nonseparable, Lindel\"{o}f GO-space.  The space $X$ does not have an \omegaop\ base.
\end{lemma}
\begin{proof}
Let $\mcA$ be a base of $X$.  Let us show that $\mcA$ is not \omegaop.  First, let us construct sequences of open sets $\la A_{n,k}\ra_{n,k<\omega}$ and $\la B_{n,k}\ra_{n,k<\omega}$.  Our requirements are that $B_{n,i}\subseteq A_{n,i}\in\mcA$, that $B_{n,i}$ is a countable union of open convex sets, that $\{B_{n,k}:k<\omega\}$ is a locally finite cover of $X$ and pairwise $\subseteq$\nbd-incomparable, and that $\{A_{i,k}:k<\omega\}\cap\{A_{j,k}:k<\omega\}\subseteq[X]^1$ for all $i<j<\omega$ and $n<\omega$.

Suppose $n<\omega$ and we are given $\la A_{m,k}\ra_{k<\omega}$ and $\la B_{m,k}\ra_{k<\omega}$ for all $m<n$ and they meet our requirements.  Let $p\in X$.  Set $V_p=\bigcap\{B_{m,k}:m<n\text{ and }k<\omega\text{ and }p\in B_{m,k}\}$.  The set $V_p$ is open.  If $\card{V_p}=1$, then set $U_p=V_p$.  If $\card{V_p}>1$, then choose $U_p\in\mcA$ such that $p\in U_p\subsetneq V_p$.  Set $\mcU=\{U_p:p\in X\}$.  By Lemma~\ref{LEMlocfinlindeloflots}, there exists a countable, locally finite refinement $\mcB_n$ of $\mcU$ consisting only of countable unions of open convex sets.  Since $\mcB_n$ is locally finite, it has no infinite ascending chains; hence, \wma\ $\mcB_n$ is pairwise $\subseteq$\nbd-incomparable because we may shrink $\mcB_n$ to its maximal elements.  Let $\{B_{n,k}:k<\omega\}=\mcB_n$.  For each $k<\omega$, set $A_{n,k}=U_p$ for some $p\in X$ satisfying $B_{n,k}\subseteq U_p$.  Suppose $m<n$ and $i,j<\omega$ and $A_{m,i}=A_{n,j}\not\in[X]^1$.  Choose $p\in X$ such that $A_{n,j}=U_p$; choose $k<\omega$ such that $p\in B_{m,k}$.  We then have $B_{m,i}\subseteq A_{m,i}=U_p\subsetneq V_p\subseteq B_{m,k}$, in contradiction with the pairwise $\subseteq$\nbd-incomparability of $\{B_{m,l}:l<\omega\}$.  Thus, $\{A_{m,l}:l<\omega\}\cap\{A_{n,l}:l<\omega\}\subseteq[X]^1$ for all $m<n$.  By induction, $\la A_{n,k}\ra_{n,k<\omega}$ and $\la B_{n,k}\ra_{n,k<\omega}$ meet our requirements.

Let $\{X,\leq,\mcA\}\subseteq M\elemsub H_\theta$ and $\card{M}=\omega$.  Since $X$ is nonseparable, there must be a nonempty open convex set $W$ disjoint from $M$.  Since $X$ is Lindel\"of, every isolated point of $X$ is in $M$.  Hence, $W$ must be infinite. Choose $\{a<c<e\}\subseteq W$.  Choose $U\in\mcA$ such that $U\subseteq(a,e)$.  By elementarity, \wma\ that $\la\la A_{n,k},B_{n,k}\ra\ra_{n,k<\om}\in M$ for all $n,k<\omega$.  For each $n<\omega$, choose $i_n<\omega$ such that $c\in B_{n,i_n}$.  To complete the proof, it suffices to show that $U$ has infinitely many supersets in $\mcA$.  Fix $n<\omega$.  Since $c\not\in M$, we cannot have $A_{n,i_n}=\{c\}$; hence, $A_{n,i_n}\not=A_{m,i_m}$ for all $m<n$.  Hence, it suffices to show that $U\subseteq A_{n,i_n}$. There exists $\la I_j\ra_{j<\omega}\in M$ such that $B_{n,i_n}=\bigcup_{j<\omega}I_j$ and each $I_j$ is convex and open.  Hence, there exists $j<\omega$ such that $c\in I_j$.  Since $U\subseteq (a,e)$ and $I_j\subseteq B_{n,i_n}\subseteq A_{n,i_n}$, it suffices to show that $(a,e)\subseteq I_j$.  Seeking a contradiction, suppose $(a,e)\not\subseteq I_j$. Since $c\in I_j$, we must have $a<b<I_j$ for some $b\in X$ or $I_j<d<e$ for some $d\in X$.  By symmetry, \wma\ we have the former.  Since $X$ is Lindel\"of, $\{p\in X: p<I_j\}$ has a countable cofinal subset $Y$.  Since $I_j\in M$, \wma\ $Y\in M$.  Hence, $Y\subseteq M$; hence, there exists $y\in M$ such that $b\leq y<I_j$.  Hence, $M$ intersects $(a,e)$; hence, $M$ intersects $W$, which is absurd.
\end{proof}

\begin{theorem}\label{THMlindeloflots}
Let $X$ be a Lindel\"{o}f GO-space.  The following are equivalent.
\begin{enumerate}
\item\label{enumXmetric} $X$ is metrizable.
\item\label{enumXomegaopbase} $X$ has an \omegaop\ base.
\item\label{enumXsepomega1opbase} $X$ is separable and has an \cardop{\omega_1}\ base.
\end{enumerate}
\end{theorem}
\begin{proof}
By Theorem~\ref{THMallmetric}, (\ref{enumXmetric}) implies (\ref{enumXomegaopbase}).  By Lemma~\ref{LEMnonseplots}, (\ref{enumXomegaopbase}) implies (\ref{enumXsepomega1opbase}).  Hence, it suffices to show that (\ref{enumXsepomega1opbase}) implies (\ref{enumXmetric}).  Suppose $X$ has a countable dense subset $D$ and an \cardop{\omega_1}\ base.  We then have $\piweight{X}=\omega$; hence, by Proposition~\ref{PROowpiw}, $\weight{X}=\omega$; hence, $X$ is metrizable.
\end{proof}

See Example~\ref{EXlindom1} for a nonseparable Lindel\"of linear order that has Noetherian type $\oml$.

\section{Small Noetherian types and smaller densities}

For compact \lots s, the theorem at the end of this section strengthens Theorem~\ref{THMlindeloflots}.  To prepare for this theorem, we first prove our main technical lemma, which we state in very general terms.

\begin{lemma}\label{LEMmucpctlotsdensity}
Suppose $\kappa$ is a regular uncountable cardinal, $\mu$ is an infinite cardinal, $\card{\lambda^{<\mu}}<\kappa$ for all $\lambda<\kappa$, $X$ is a product of fewer than $\kappa$\many\ $\mu$\nbd-compact GO-spaces, and $\ow{X}\leq\kappa$.  We then have $\density{X}<\kappa$.
\end{lemma}
\begin{proof}
Let $X=\prod_{i<\nu}X_i$ where $\nu<\kappa$ and each $X_i$ is a $\mu$\nbd-compact subspace of a \lots\ $Y_i$.  Seeking a contradiction, suppose $\density{X}\geq\kappa$.  Let $\mcU$ be a \kappaop\ base of $X$, $\{\la Y_i,\leq_{Y_i},X_i:i<\nu\ra,\mcU\}\in M\elemsub H_\theta$, $\card{M}<\kappa$, $M\cap\kappa\in\kappa$, and $M^{<\mu}\subseteq M$.  (We can construct $M$ as the union of an appropriate elementary chain of length $\rho$, where $\rho$ is the least regular cardinal $\geq\mu$.  Such an $M$ is not too large because $\rho<\kappa$, a fact that follows from $\mu\leq\card{2^{<\mu}}<\kappa$ and $\cf(\mu)<\mu\Rightarrow\mu^+\leq\card{\mu^{\cf(\mu)}}<\kappa$).  Since $\density{X}>\card{M}$, there is a finite subproduct $\prod_{i\in\sigma}X_i$ of $X$ that has a nonempty open subset disjoint from $M$.  We may choose this open subset to be the interior of a set of the form $B=\prod_{i\in\sigma}B_i$ where each $B_i$ is maximal among the convex subsets of $X_i$ disjoint from $M$.  Set $A_i=\{p\in X_i: p<B_i\}$ and $C_i=\{p\in X_i: p>B_i\}$.  Since $\{p:A_i<p<C_i\}=B_i$, which is nonempty but disjoint from $M$, we have $\{A_i,C_i\}\not\subseteq M$ by elementarity.

\begin{claim}
$\max\{\cf(A_i),\ci(C_i)\}\geq\mu$ for all $i\in\sigma$.  
\end{claim}
\begin{proof}
Seeking a contradiction, suppose $\cf(A_i)<\mu$ and $\ci(C_i)<\mu$.  
We then have $A_i\cap M,C_i\cap M\in M$.  Since $A_i\cap M$ is cofinal in $A_i$, $A_i=\{p:\exists q\in A_i\cap M\ \,p\leq q\}\in M$.  Likewise, $C_i\in M$, in contradiction with the fact that $\{A_i,C_i\}\not\subseteq M$.
\end{proof}

Therefore, \wma\ that $\cf(A_i)\geq\mu$ for all $i\in\sigma$ (by symmetry).  Since $X$ is $\mu$\nbd-compact, there exists $x_i=\sup_{Y_i}(A_i)=\min(B_i)\in X_i$ for all $i\in\sigma$.  

\begin{claim}
There exists $y_i=\sup_{Y_i}(B_i)\in(x_i,\infty)$ for all $i\in\sigma$, with the understanding that in this proof all intervals are intervals of $X_i$ (so $y_i\in X_i$).
\end{claim}
\begin{proof}
If $\ci(C_i)\geq\mu$, then, by $\mu$\nbd-compactness, there exists $y_i=\inf_{Y_i}(C_i)\in X_i$.  In this case, $y_i$ is also $\max(B_i)$ because $y_i\not\in C_i$.  Moreover, $y_i=\max(B_i)>\min(B_i)=x_i$ because otherwise the interior of $B$ would be empty, for $x_i=\sup_{Y_i}(A_i)$, which is not an isolated point in $X_i$.  If $\ci(C_i)<\mu$, then $C_i\in M$, just as in the previous claim's proof, so there exists $D_i\in M$ such that $D_i$ is a cofinal subset of $\{p\in X_i: p<C_i\}$ of minimal size.  In this case, $D_i$ includes a cofinal subset of $B_i$, so $D_i\not\subseteq M$, so $\card{D_i}\geq\kappa$, so $\mu<\kappa\leq\card{D_i}=\cf(B_i)$, so there exists $y_i=\sup_{Y_i}(B_i)\in X_i$ by $\mu$\nbd-compactness.  Also, $\cf(B_i)\geq\kappa$ implies $\sup_{Y_i}(B_i)>\min(B_i)=x_i$.  Thus, in any case there exists $y_i=\sup(B_i)\in(x_i,\infty)$ for all $i\in\sigma$.
\end{proof}

Let $U\in\mcU$ satisfy $x_i\in \pi_i[U]\subseteq(-\infty,y_i)$ for all $i\in\sigma$.  Since $\cf(A_i)\geq\mu>1$ for all $i\in\sigma$, we then have $\la x_i:i\in\sigma\ra\in\prod_{i\in\sigma}W_i\subseteq \pi_\sigma[U]$ where each $W_i$ is of the form $(u_i,v_i)$ or $(u_i,v_i]$ for some $u_i<x_i$ and $v_i\leq y_i$; \wma\ $u_i\in M$.  Moreover, there exist $p_i,q_i\in X_i\cap M$ such that $u_i<p_i<q_i<x_i$.  Since $\mcU$ is a \kappaop\ base, it includes fewer than $\kappa$\many\ supersets of $\bigcap_{i\in\sigma}\pi_i^{-1}[(u_i,q_i)]$ as members.  Since the set of supsersets of $\bigcap_{i\in\sigma}\pi_i^{-1}[(u_i,q_i)]$ in $\mcU$ is a set in $M$ and a set of size less than $\kappa$, it is also a subset of $M$.  In particular, $U\in M$.  

Fix an arbitrary $i\in\sigma$.  If $\cf(\pi_i[U])<\mu$, then $M$ would include a cofinal subset of $\pi_i[U]$, in contradiction with $B_i$ missing $M$.  Therefore, $\cf(\pi_i[U])\geq\mu$.  Hence, there exists $z=\sup_{Y_i}(\pi_i[U])$.  By elementarity, $z\in M$, so $B_i<z$, so $z=\min(C_i)=\sup_{Y_i}(B_i)=y$.  Because of the freedom in how we chose $U$, it follows that every neighborhood of $x_i$ includes a neighborhood that, like $\pi_i[U]$, has supremum $y_i$ (in $Y_i$) and has cofinality at least $\mu$.  Therefore, there is an infinite increasing sequence of points between $x_i$ and $y_i$ that are contained in every neighborhood of $x_i$, in contradiction with $X_i$ being a subspace of the ordered space $Y_i$.  Thus, $\density{X}<\kappa$.
\end{proof}

\begin{corollary}
Suppose that $X$ is a product of at most $2^{\aleph_0}$\many\ Lindel\"of GO-spaces such that $\ow{X}\leq\left(2^{\aleph_0}\right)^+$.  We then have $\density{X}\leq 2^{\aleph_0}$.
\end{corollary}
\begin{corollary}\label{CORcpctlotsdensity}
Suppose that  $\kappa$ is a regular uncountable cardinal, $X$ is a product of less than $\kappa$\many\ \locpcta, and $\ow{X}\leq\kappa$.  We then have $\density{X}<\kappa$.
\end{corollary}

The last corollary fails for singular $\kappa$.  As we shall see in Theorem~\ref{THMkappalots}, if $\lambda$ is an uncountable singular cardinal, then $\ow{\lambda+1}=\lambda$, despite the fact that $\density{\kappa+1}=\kappa$ for all infinite cardinals $\kappa$.  Moreover, the space $(\kappa+1)^\kappa$ has Noetherian type $\om$ and density $\kappa$ for all infinite cardinals $\kappa$, so we cannot weaken the above hypothesis that $X$ has less than $\kappa$\many\ factors.  The equation $\ow{(\kappa+1)^\kappa}=\om$ follows from a general theorem of Malykhin.

\begin{theorem}\label{THMowproductweight}\cite{malykhin}
Let $X=\prod_{i\in I}X_i$ where each $X_i$ has a minimal open cover of size two (\eg\ $X_i$ is $T_1$).  If $\sup_{i\in I}\weight{X_i}\leq\card{I}$, then $\ow{X}=\omega$.
\end{theorem}
\begin{proof}
For each $i\in I$, let $\{U_{i,0},U_{i,1}\}$ be a minimal open cover of $X_i$.  Since $\weight{X}=\sup_{i\in I}\weight{X_i}$, we may choose $\mcA$ to be a base of $X$ of size at most $\card{I}$ and choose an injection $f\colon \mcA\rightarrow I$.  Let $\mcB$ denote the set of all nonempty sets of the form $V\cap\inv{\pi}_{f(V)}\left[U_{f(V),j}\right]$ where $V\in\mcA$ and $j<2$.  Since $f$ is injective, every infinite subset of $\mcB$ has empty interior.  Hence, $\mcB$ is an \omegaop\ base of $X$.
\end{proof}

\begin{theorem}\label{THMcpctlots}
Let $X$ be a product of countably many \locpcta.  The following are equivalent.
\begin{enumerate}
\item\label{enumcpctXmetric} $X$ is metrizable.
\item\label{enumcpctXomegaopbase} $X$ has an \omegaop\ base.
\item\label{enumcpctXomega1opbase} $X$ has an \cardop{\omega_1}\ base.
\item\label{enumcpctXsepomega1opbase} $X$ is separable and has an \cardop{\omega_1}\ base.
\end{enumerate}
\end{theorem}
\begin{proof}
By Theorem~\ref{THMallmetric}, (\ref{enumcpctXmetric}) implies (\ref{enumcpctXomegaopbase}), which trivially implies (\ref{enumcpctXomega1opbase}).  By Corollary~\ref{CORcpctlotsdensity}, (\ref{enumcpctXomega1opbase}) implies (\ref{enumcpctXsepomega1opbase}).  Finally, (\ref{enumcpctXsepomega1opbase}) implies (\ref{enumcpctXmetric}) because if $X$ is separable, then $\piweight{X}=\om$, so $\weight{X}=\om$ by Proposition~\ref{PROowpiw}.
\end{proof}

\section{The Noetherian spectrum of the compact orders}\label{SECspectra}

Theorem~\ref{THMcpctlots} implies that no \locpct\ has Noetherian type $\omega_1$.  What is the class of Noetherian types of \locpcta?  We shall prove that an infinite cardinal $\kappa$ is the Noetherian type of a \locpct\ if and only if $\kappa\not=\omega_1$ and $\kappa$ is not weakly inaccessible.

\begin{theorem}\label{THMkappalots}
Let $\kappa$ be an uncountable cardinal and give $\kappa+1$ the order topology.  If $\kappa$ is regular, then $\ow{\kappa+1}=\kappa^+$; otherwise,  $\ow{\kappa+1}=\kappa$.
\end{theorem}
\begin{proof}
Using Corollary~\ref{CORcpctlotsdensity}, the lower bounds on $\ow{\kappa+1}$ are easy.  We have $\density{\kappa+1}\geq\lambda$ for all regular $\lambda\leq\kappa$, so $\ow{\kappa+1}>\lambda$ for all regular $\lambda\leq\kappa$.  It follows that $\ow{\kappa+1}\geq\kappa$ and $\ow{\kappa+1}>\cf\kappa$.  We can also prove these lower bounds directly using the Pressing Down Lemma.  Let $\mcA$ be a base of $\kappa+1$ and let $\lambda$ be a regular cardinal $\leq\kappa$.  Let us show that $\mcA$ is not \cardop{\lambda}.  For every limit ordinal $\alpha<\lambda$, choose $U_\alpha\in\mcA$ such that $\alpha=\max U_\alpha$; choose $\eta(\alpha)<\alpha$ such that $[\eta(\alpha),\alpha]\subseteq U_\alpha$.  By the Pressing Down Lemma, $\eta$ is constant on a stationary subset $S$ of $\lambda$.  Hence, $\mcA\ni\{\eta(\min S)+1\}\subseteq U_\alpha$ for all $\alpha\in S$; hence, $\mcA$ is not \cardop{\lambda}.  Once again, it follows that $\ow{\kappa+1}\geq\kappa$ and $\ow{\kappa+1}>\cf\kappa$.  

Trivially, $\ow{\kappa+1}\leq\weight{\kappa+1}^+=\kappa^+$.  Hence, it suffices to show that $\kappa+1$ has a \kappaop\ base if $\kappa$ is singular.  Suppose $E\in[\kappa]^{<\kappa}$ is unbounded in $\kappa$.  Let $F$ be the set of limit points of $E$ in $\kappa+1$.  Define $\mcB$ by
\begin{equation*}
\mcB=\{(\beta,\alpha]:E\ni\beta<\alpha\in F\text{ or }\sup(E\cap\alpha)\leq\beta<\alpha\in\kappa\setminus F\}.
\end{equation*}
The set $\mcB$ is a \kappaop\ base of $\kappa+1$.
\end{proof}


\begin{definition}
Given a poset $P$ with ordering $\leq$, let $P^\op$ denote the set $P$ with ordering $\geq$.
\end{definition}

\begin{theorem}
Suppose $\kappa$ is a singular cardinal.  There then is a \locpct\ with Noetherian type $\kappa^+$.
\end{theorem}
\begin{proof}
Set $\lambda=\cf\kappa$ and $X=\lambda^++1$.  Partition the set of limit ordinals in $\lambda^+$ into $\lambda$\many\ stationary sets $\la S_\alpha\ra_{\alpha<\lambda}$.  Let $\la\kappa_\alpha\ra_{\alpha<\lambda}$ be an increasing sequence of regular cardinals with supremum $\kappa$.  For each $\alpha<\lambda$ and $\beta\in S_\alpha$, set $Y_\beta={(\kappa_\alpha+1)}^\op$.  For each $\alpha\in X\setminus\bigcup_{\beta<\lambda}S_\beta$, set $Y_\alpha=1$.  Set $Y=\bigcup_{\alpha\in X}\{\alpha\}\times Y_\alpha$ ordered lexicographically.  We then have $\ow{Y}\leq\weight{Y}^+\leq\card{Y}^+=\kappa^+$.  Hence, it suffices to show that $Y$ has no \kappaop\ base.

Seeking a contradiction, suppose $\mcA$ is a \kappaop\ base of $Y$.  For each $\alpha<\lambda$, let $\mcU_\alpha$ be the set of all $U\in\mcA$ that have at least $\kappa_\alpha$\many\ supersets in $\mcA$.  For all isolated points $p$ of $Y$, there exists $\alpha<\lambda$ such that $\{p\}\not\in\mcU_\alpha$; whence, $p\not\in\bigcup\mcU_\alpha$.  Since $\la\alpha+1,0\ra$ is isolated for all $\alpha<\lambda^+$, there exist $\beta<\lambda$ and a set $E$ of successor ordinals in $\lambda^+$ such that $\card{E}=\lambda^+$ and $(E\times 1)\cap\bigcup\mcU_\beta=\varnothing$.  Let $C$ be the closure of $E$ in $\lambda^+$.  The set $C$ is closed unbounded; hence, there exists $\gamma\in C\cap S_{\beta+1}$.  Set $q=\la\gamma,\kappa_{\beta+1}\ra$.  We then have $q\in\closure{E\times 1}$; hence, $q\not\in\bigcup\mcU_\beta$.  Since $q$ has coinitiality $\kappa_{\beta+1}$, any local base $\mcB$ at $q$ will contain an element $U$ such that $U$ has $\kappa_\beta$\many\ supersets in $\mcB$.  Hence, there exists $U\in\mcU_\beta$ such that $q\in U$; hence, $q\in\bigcup\mcU_\beta$, which yields our desired contradiction.
\end{proof}

\begin{theorem}\label{THMlcownotwi}
No \locpct\ has weakly inaccessible Noetherian type.  More generally, for every weakly inaccessible $\kappa$, products of fewer than $\kappa$\many\ \locpcta\ do not have Noetherian type $\kappa$.
\end{theorem}
\begin{proof}
Suppose $\kappa$ is weakly inaccessible, $X$ is a product of fewer than $\kappa$\many\ \locpcta, and $\ow{X}\leq\kappa$.  It suffices to prove $\ow{X}<\kappa$.  By Corollary~\ref{CORcpctlotsdensity}, we have $\density{X}<\kappa$; hence, each factor of $X$ has $\pi$\nbd-weight less than $\kappa$; hence, $\piweight{X}<\kappa$.  If $\weight{X}\geq\kappa$, then $\ow{X}>\kappa$ by Proposition~\ref{PROowpiw}, in contradiction with our assumptions about $X$.  Hence, $\weight{X}<\kappa$; hence, $\ow{X}\leq\weight{X}^+<\kappa$.
\end{proof}

\section{The Lindel\"of spectrum}

The spectrum of Noetherian types of Lindel\"of \lots s trivially includes the spectrum of Noetherian types of compact \lots s.  More interestingly, the inclusion is strict, as the next example shows.

\begin{example}\label{EXlindom1}
Theorem~\ref{THMcpctlots} fails for Lindel\"{o}f \lots s.  Let $X$ be $(\omega_1\times\integers)\cup(\{\omega_1\}\times\{0\})$ ordered lexicographically.  The space $X$ is Lindel\"{o}f and nonseparable and $\{\{\la\alpha,n\ra\}:\alpha<\omega_1\text{ and }n\in\integers\}\cup\{X\setminus(\alpha\times\integers):\alpha<\omega_1\}$ is an \cardop{\omega_1}\ base of $X$.  Moreover, $X$ has no \omegaop\ base because every local base at $\la\oml,0\ra$ includes a descending $\oml$\nbd-chain of neighborhoods.  Thus, $\ow{X}=\oml$.  Easily generalizing this example, if $\kappa$ is a regular cardinal and $X$ is $(\kappa\times\integers)\cup(\{\kappa\}\times\{0\})$ ordered lexicographically, then $X$ is $\kappa$\nbd-compact and $\ow{X}=\kappa$.
\end{example}

A consequence of Lemma~\ref{LEMmucpctlotsdensity} is that Lindel\"of \lots s cannot have strongly inaccessible Noetherian type, just as in the compact case.  More generally, we have the following theorem, which is proved just as Theorem~\ref{THMlcownotwi} was proved.

\begin{theorem}\label{THMmucpctlotsntnotk}
Suppose $\kappa$ is a weakly inaccessible cardinal, $\card{\lambda^{<\mu}}<\kappa$ for all $\lambda<\kappa$, and $X$ is a $X$ is a product of fewer than $\kappa$\many\ $\mu$\nbd-compact GO-spaces.  We then have $\ow{X}\not=\kappa$.
\end{theorem}
\begin{proof}
Suppose that $\ow{X}\leq\kappa$.  Let us show that $\ow{X}<\kappa$. By Lemma~\ref{LEMmucpctlotsdensity}, we have $\density{X}<\kappa$; hence, each factor of $X$ has $\pi$\nbd-weight less than $\kappa$; hence, $\piweight{X}<\kappa$.  If $\weight{X}\geq\kappa$, then $\ow{X}>\kappa$ by Proposition~\ref{PROowpiw}, in contradiction with our assumptions about $X$.  Hence, $\weight{X}<\kappa$; hence, $\ow{X}\leq\weight{X}^+<\kappa$.
\end{proof}

\begin{corollary}
If $\kappa$ is strongly inaccessible, then the class of Noetherian types of $\mu$\nbd-compact GO-spaces excludes $\kappa$ if and only if $\mu<\kappa$.
\end{corollary}
\begin{proof}
``If'': Theorem~\ref{THMmucpctlotsntnotk}.
``Only if'': Example~\ref{EXlindom1}.
\end{proof}

On the other hand, it is consistent (relative to the consistency of an inaccessible), that some Lindel\"of \lots\ has weakly inaccessible Noetherian type.  To show this, we first force $2^{\aleph_0}\geq\kappa$ where $\kappa$ is weakly inaccessible (say, by adding $\kappa$\many\ Cohen reals).  Next, we construct the desired linear order in this forcing extension using the following theorem.

\begin{theorem}
If $\kappa$ is a weak inaccessible and $2^{\aleph_0}\geq\kappa$, then there is a Lindel\"of linear order $Z$ such that $\ow{Z}=\kappa$.
\end{theorem}
\begin{proof}
Let $B$ be a Bernstein subset of $X=[0,1]$, \ie\ $B$ includes some point in $P$ and misses some point in $P$, for all perfect $P\subseteq X$.  Let $f\colon B\rightarrow\kappa$ be surjective.  For each $x\in B$, set $Y_x=\om^{\text{op}}+\om_{f(x)}+\om$, which is Lindel\"of.  For each $x\in X\setminus B$, set $Y_x=\{0\}$.  Set $Z=\bigcup_{x\in X}(\{x\}\times Y_x)$ ordered lexicographically.  First, let us show that $Z$ is Lindel\"of.  Let $\mcU$ be an open cover of $Z$.  For every $x\in X\setminus B$, $\la x,0\ra$ has neighborhoods $O_x$ and $U_x$ such that $\mcU\ni U_x\supseteq O_x=\bigcup_{a<b<c}(\{b\}\times Y_b)$ where $a,c\in (X\cap\rationals)\cup\{\pm\infty\}$.  
Therefore, there is a countable $D\subseteq X$ such that $\{O_x:x\in D\}$ covers $(X\setminus B)\times\{0\}$.  Set $\mcV=\{U_x:x\in D\}$ and
$C=\{x\in X: Y_x\not\subseteq\bigcup_{x\in D}O_x\}$.  The set $C$ is closed in $X$ and a subset of the Bernstein set $B$, so $C$ is countable.  Therefore, $\bigcup_{x\in C}(\{x\}\times Y_x)$ is Lindel\"of; hence, it is covered by a countable $\mcW\subseteq\mcU$, making $\mcV\cup\mcW$ a countable subcover of $\mcU$.

Finally, let us show that $\ow{Z}=\kappa$.  For every $\alpha<\kappa$ and $x\in\inv{f}[\{\alpha+1\}]$, $Y_x$ has a point with cofinality $\om_{\alpha+1}$, so $\ow{Z}\geq\omega_{\alpha+1}$.  Therefore, it suffices to construct a \kappaop\ base of $Z$.  Let $\mcA$ denote the countable set of all sets of the form $\bigcup_{a<b<c}Y_b$ where $a,c\in (X\cap\rationals)\cup\{\pm\infty\}$, which includes a local base at $\la x,0\ra$ for every $x\in X\setminus B$.  Since each $Y_x$ for $x\in B$ has no maximum or minimum, we can combine $\mcA$ with a copy of a base $\mcB_x$ of $Y_x$ for each $x\in B$ in order to produce a base $\mcB$ of $Z$.  We may choose each $\mcB_x$ to have size less than $\kappa$, so $\mcB$ must be \kappaop.
\end{proof}

\end{document}